\documentclass[11pt]{article}
\usepackage{amssymb}
\usepackage{graphicx}
\usepackage{setspace}
  \usepackage{paralist}
  \usepackage{longtable}
   \usepackage{multirow}
    \usepackage{rotating}
\usepackage{marginnote}
\usepackage{mathrsfs}
\usepackage{amsmath, amsthm, amssymb}
\usepackage{authblk}
\usepackage{graphicx}
\usepackage{float}
\usepackage{hyperref}
\usepackage[margin=1in]{geometry}
\usepackage{comment}
\usepackage{color}
\usepackage{cite}
\usepackage{indentfirst}
\allowdisplaybreaks[4]
\numberwithin{equation}{section}
\date{}
\newtheorem{theorem}{Theorem}[section]
\newtheorem*{theorem*}{Theorem}

\newtheorem{definition}[theorem]{Definition}

\theoremstyle{remark}

\begin{document}
% \renewcommand{\thefootnote}{\fnsymbol{footnote}}

% \footnotetext{School of Mathematical Sciences, CMA-Shanghai, Shanghai Jiao Tong University, China;}
% \footnotetext{Authors are partially supported by NSFC-12031012, NSFC-11831003 and the Institute of Modern Analysis-A Frontier Research Center of Shanghai.}
% \renewcommand{\thefootnote}{\arabic{footnote}}

\title{Direct Method of Scaling Spheres for the Laplacian and Fractional Laplacian Equations with Hardy-H\'enon Type Nonlinearity}

\author{Meiqing Xu\footnote{School of Mathematical Sciences, CMA-Shanghai, Shanghai Jiao Tong University, China. The author is partially supported by NSFC-12031012, NSFC-11831003 and the Institute of Modern Analysis-A Frontier Research Center of Shanghai.}}

\maketitle
\begin{abstract}
% In this paper, we apply the narrow region principle and the direct method of scaling spheres investigated by Dai and Qin (\cite{dai2023liouville}, \textit{International Mathematics Research Notices, 2023}) to the Laplacian and fractional Laplacian equations with Hardy-H\'enon type nonlinearity.
In this paper, we focus on the partial differential equation
\begin{equation*}
    (-\Delta)^\frac{\alpha}{2} u(x)=f(x,u(x))\;\;\;\;\text{ in }\mathbb{R}^n,
\end{equation*}
where $0<\alpha\leq 2$. By the direct method of scaling spheres investigated by Dai and Qin (\cite{dai2023liouville}, \textit{International Mathematics Research Notices, 2023}), we derive a Liouville-type theorem. This mildly extends the previous researches on Liouville-type theorem for the semi-linear equation $ (-\Delta)^\frac{\alpha}{2} u(x)=f(u(x))$ where the nonlinearity $f$ depends solely on the solution $u(x)$, and covers the Liouville-type theorem for Hardy-H\'enon equations $(-\Delta)^\frac{\alpha}{2} u(x)=|x|^au^p(x)$.


\end{abstract}

%\tableofcontents
\section{Introduction}
In this paper we consider the partial differential equation
\begin{equation}\label{main PDE}
    (-\Delta)^\frac{\alpha}{2} u(x)=f(x,u(x))\;\;\;\;\text{ in }\mathbb{R}^n,
\end{equation}
where $n>\alpha$, $0<\alpha\leq 2$, and the Hardy-H\'enon type nonlinearity $f:(\mathbb{R}^n\backslash \{0\}) \times [0,+\infty)\to [0,+\infty)$.
For $\alpha=2$, $(-\Delta)^\frac{\alpha}{2}$ is the classical Laplacians operator, that is, $-\Delta = -\sum_{i=1}^n \frac{\partial^2}{\partial x_i^2}$. For $\alpha$ taking any real number between 0 and 2, $(-\Delta)^\frac{\alpha}{2}$ is a nonlocal differential operator defined by
\begin{equation}\label{fraction definition}
  (-\Delta)^\frac{\alpha}{2}u(x)=C_{n,\alpha}P.V. \int_{\mathbb{R}^n} \frac{u(x)-u(y)}{|x-y|^{n+\alpha}}dy.
\end{equation}
where P.V. is the Cauchy principal value. The integral on the right-hand
side of (\ref{fraction definition}) is well defined for $u\in C_{loc}^{[\alpha],\{\alpha\}+\epsilon}(\mathbb{R}^n) \cap\mathcal{L}_\alpha(\mathbb{R}^n)$, where $\epsilon>0$ is arbitrary small, $[\alpha]$ denotes the integer part of $\alpha$, $\{\alpha\}=\alpha-[\alpha]$,
\begin{equation}
  \mathcal{L}_\alpha(\mathbb{R}^n)=\{u:\mathbb{R}^n\rightarrow\mathbb{R}| \int_{\mathbb{R}^n} \frac{|u(x)|}{(1+|x|^{n+\alpha})}dx <\infty\}.
\end{equation}
For more details, see \cite{silvestre,chen2020fractional}.

For $\alpha=2$, we assume the solution $u\in C^2(\mathbb{R}^n\backslash\{0\})\cap C(\mathbb{R}^n)$. For $0<\alpha<2$, we assume the solution $u\in  C_{loc}^{[\alpha],\{\alpha\}+\epsilon}(\mathbb{R}^n\backslash\{0\}) \cap\mathcal{L}_\alpha(\mathbb{R}^n)$.

% \begin{definition}
%     We say that the non-linear term $f$ has subcritical growth provided that
%     \begin{equation*}
%         \mu^\frac{n+\alpha}{n-\alpha}f(\mu^\frac{2}{n-\alpha}x,\mu^{-1}u)
%     \end{equation*}
%     is strictly increasing when $\mu\geq 1$ for all $(x,\mu)\in(\mathbb{R}^n\backslash \{0\})\times \mathbb{R}^n_+$.
% \end{definition}

\begin{definition}
    A function $g(x,u)$ is called locally Lipschitz on $u$ in $\mathbb{R}^n\times [0,+\infty)$, provided that for any $u_0\in [0,+\infty)$ and $\Omega\subset\mathbb{R}^n$ bounded, there exists a (relatively) open neighborhood $U(u_0)\subset[0,+\infty)$ such that $g$ is Lipschitz continuous on $u$ in $\Omega\times U(u_0)$.
\end{definition}

We need the following assumptions about the Hardy-H\'enon type nonlinearity $f(x,u)$.
\begin{enumerate}
    \item [$(f_0)$] The nonlinearity $f$ has subcritical growth, i.e.
    \begin{equation*}
        f(t^{-\frac{2}{n-\alpha}}x,tu(x))/{t^\frac{n+\alpha}{n-\alpha}}>f(x,u(x))\;\;\;\; \text{for any $t> 1$ and $(x,u)\in(\mathbb{R}^n\backslash \{0\})\times \mathbb{R}^n_+$}.
    \end{equation*} 
    \item [$(f_1)$] The nonlinearity $f$  is nondecreasing about $u$ in $(\mathbb{R}^n\backslash \{0\}) \times \mathbb{R}^n_+$.
    \item[$(f_2)$] There exists a $\sigma<\alpha$ such that $|x|^\sigma f(x,u)$ is locally Lipschitz on $u$ in $\mathbb{R}^n\times [0,+\infty)$.
    \item[$(f_3)$] There exists a cone $\mathcal{C}\subset\mathbb{R}^n$ with vertex at 0, constants $C>0$, $-\alpha<a<+\infty$, $0<p<p_c(a)$ such that
    \begin{equation*}
        f(x,u)\geq C|x|^a u^p(x)\;\;\;\;\text{in $\mathcal{C}\times [0,+\infty)$.}
    \end{equation*} 
\end{enumerate}

The Liouville type theorems of the Lane-Emden equations, Hardy-H\'enon equations and equations with general nonlinearities have been extensively studied (see \cite{MR4577475,MR4350191,MR4129482,MR3528525,MR3485438,MR3431273, cao2013liouville,chen2013super2, chen2015liouville,zhuo2019liouville,cheng2016liouville, fang2012liouville}). In the study of Liouville-type theorems concerning the nonexistence of non-trivial solutions, the most commonly employed approach is the method of moving planes. This method has been extensively explored in prior works such as \cite{caffarelli,berestycki1988monotonicity,chen1997priori,chenw,li} and further developed in recent research contributions, including \cite{ouyang2023priori,MR3600062,cheng2017direct,chen2,chen4,chenw,dai2020liouville2}. However, the method of moving planes has limitations. For example, it requires the solution to be bounded or integrable.
In \cite{dai2023liouville}, Dai and Qin introduced the method of scaling spheres, which offers significant advantages. Notably, in the fractional case, it does not necessitate any boundedness or integrability of the solution, in contrast to the method of moving planes. This variation can potentially provide novel perspectives for exploring Liouville-type theorems. Another distinct feature is that the potential term in nonlinearity does not need to be strictly decreasing after the Kelvin transform, as required by the method of moving planes. This extension greatly broadens the range of the exponent $p$, from $1<p<\frac{p+a+\alpha}{n-\alpha}$ to $0<p<\frac{p+2a+\alpha}{n-\alpha}$. 
To be precise, by the direct method of scaling spheres, Dai and Qin expanded the range of $p$ in the following ways (only relevant parts of the results are listed here):
\begin{enumerate}
    \item For nonnegative solution $u$ of the equation 
    \begin{equation*}
        (-\Delta) ^\frac{\alpha}{2} u(x)=|x|^a u^p(x),\;\;\;\; x\in\mathbb{R}^n,
    \end{equation*}
    where $0<\alpha\leq 2$, $n>\alpha$, $-\alpha<a<+\infty$, $0<p<\frac{p+2a+\alpha}{n-\alpha}$, $u$ can only be trivial.
 
    % \item For nonnegative solution $u$ of the equation 
    % \begin{equation*}  u(x)=\int_{\mathbb{R}^n}C_{n,m}\frac{|y|^a u^p(y)}{|x-y|^{n-2m}}dy,\;\;\;\; x\in\mathbb{R}^n,
    % \end{equation*}
    % where $n\geq 3$, $m\in\mathbb{N}\cap [2,\frac{n}{2})$, $-2m< a<+\infty$, $0<p<\frac{n+2m+2a}{n-2m}$, $u$ can only be trivial.
  
    % \item For nonnegative solution $u$ of the equation
    % \begin{equation*}
    %     \begin{cases}
    %         (-\Delta)^\frac{\alpha}{2} u(x)=|x|^au^p(x),\;\;\;\; & x\in\mathbb{R}^n_+,\\
    %         u(x)=0,\;\;\;\; &x\in(\mathbb{R}^n_+)^c,
    %     \end{cases}
    % \end{equation*}
    % where $n>\alpha$, $0<\alpha\leq 2$, $-\alpha<a<+\infty$, and $1\leq p<\frac{n+\alpha+2a}{n-\alpha}$, Assume further $-1<a<+\infty$ if $\alpha=2$. Then $u$ can only be trivial.

    \item For a non-negative solution $u$ of the equation 
    \begin{equation*}
            \begin{cases}
            (-\Delta) ^\frac{\alpha}{2} u(x)=f(x,u),\;\;\;\; &x\in\mathbb{R}^n_+,\\
            u(x)=0,\;\;\;\; &x\in(\mathbb{R}^n_+)^c,
        \end{cases}
    \end{equation*}
    where $0<\alpha\leq 2$, $n>\alpha$, and Hardy-H\'enon type nonlinearity $f$ is subcritical and satisfies assumptions $(f_1)$, $(f_2)$ and $(f_3)$, $u$ can only be trivial.
\end{enumerate}
Dai and Qin also indicated that the problems addressable with the direct method of scaling spheres are not limited to the ones mentioned above. In fact, they explicitly point out that method of scaling spheres can also be employed to equation
\begin{equation*}
    (-\Delta)^\frac{\alpha}{2}u(x)=f(x,u),\;\;\;\; x\in\mathbb{R}^n.
\end{equation*}
with Hardy-H\'enon type nonlinearity $f$. By further developing their idea and applying the method of scaling spheres, we derive the following Liouville-type theorem.
\begin{theorem}\label{main thm1}
    Assume $n>\alpha$, $0<\alpha\leq 2$, $0<p_c(a):=\frac{n+\alpha+2a}{n-\alpha}$. Suppose $f$ is subcritical and satisfies the assumption $(f_1)$, $(f_2)$ and $(f_3)$, and $u$ is a nonnegative solution of (\ref{main PDE}). Then $u\equiv 0$, i.e. any nonnegative solution of (\ref{main PDE}) is trivial.
\end{theorem}
To establish Theorem \ref{main thm1}, we begin by giving the equivalence (Theorem \ref{equiv thm}) between (\ref{main PDE}) and the corresponding integral equation
\begin{equation}\label{main IE}
u(x)=C_{n,\alpha}\int_{\mathbb{R}^n} \frac{f(y,u(y))}{|x-y|^{n-\alpha}} dy.
\end{equation}
Then we establish the narrow region principle (Theorem \ref{thm narrow}), and use the (direct) method of scaling spheres.

In particular, $f(x,u)=|x|^a u^p$ with $a>-\alpha$ satisfies all the assumptions in Theorem \ref{main thm1}. Therefore, this result covers the Liouville-type theorem presented in \cite{dai2023liouville} for the equation $(-\Delta)^\frac{\alpha}{2}u(x)=|x|^a u^p$. In \cite{yu2013liouville}, Yu applied the method of moving planes in the integral form to get the Liouville-type theorem of the integral equation
\begin{equation*}
    u(x)=\int_{\mathbb{R}^n} \frac{f(u(y))}{|x-y|^{n-\alpha}} dy,\;\;\;\;x\in \mathbb{R}^n,
\end{equation*}
where $f(u)$ is continuous and non-decreasing, and $f(t)/t^\frac{n+\alpha}{n-\alpha}$ is non-decreasing. Also in \cite{chen2017direct}, Chen, Li and Zhang applied the direct method of moving spheres to get the Liouville-type theorem of the partial differential equation
\begin{equation*}
    (-\Delta)^\frac{\alpha}{2}u(x)=f(u(x)),\;\;\;\;x\in \mathbb{R}^n,
\end{equation*}
where $f(u)$ is locally bounded and $f(t)/t^\frac{n+\alpha}{n-\alpha}$ is non-decreasing. Distinguishing our research from previous works where the nonlinearity $f(u)$ depends solely on the solution $u$, we midly extend the scope by considering the spatial dependence in the nonlinearity $f(x,u(x))$. 
For additional research on fractional Laplacians with non-linearity, see \cite{dai2021liouville,dai2023liouville,ros2014pohozaev,fall2016monotonicity,fall2012nonexistence,ros2014dirichlet,fall2014unique,ros2015nonexistence,liu2022dirichlet,liu2016radial,li2017symmetry,yu2016solutions}. For further studies using the method of scaling spheres, see \cite{peng2022existence,peng2021liouville,peng2023classification,dai2019liouville,dai2020liouville,dai2021liouville2,le2021method,dai2023method,dai2022nonexistence,cao2021liouville}

As mentioned in \cite{dai2023liouville}, when dealing with fractional Laplacian, it is often more convenient to apply the Hardy-Littlewood-Sobolev inequality and the method of scaling spheres in integral form. We hope to apply the latter approach to the higher-order fractional Laplacian equations with general nonlinearity and get the Liouville-type theorem. 

The paper is organized as follows. In Section 2 we prove the equivalence between PDE (\ref{main PDE}) and IE (\ref{main IE}). In Section 3 we leverage contradiction arguments, the narrow region principle (Theorem \ref{thm narrow}) and the direct method of scaling spheres to prove Theorem \ref{main thm1}.
\section{Equivalence between PDE and IE}
We first show the equivalence between (\ref{main PDE}) and the integral equation
i.e. the solution of the PDE also solves the IE (\ref{main IE}). The idea comes from \cite{zhuo2015symmetry,dai2023liouville}.
\begin{theorem}\label{equiv thm}
   Assume $n>\alpha$, $0<\alpha\leq 2$ and $0<p<+\infty$. Suppose $f$ is subcritical and $f$ satisfies the assumption $(f_1)$, $(f_2)$ and $(f_3)$. Suppose $u$ is a non-trivial solution of (\ref{main PDE}). Then it also solves the IE (\ref{main IE}), and vice versa.
\end{theorem}
One can check that if $u$ solves the IE then it solves the PDE. Thus we only check that any solution of the PDE is also a solution of IE. The idea of proof basically comes from \cite{dai2023liouville}, Theorem 2.1.
\begin{proof}
    For any $R>0$, let
    \begin{equation*}
        v_R(x)=\int_{B_R}G_R^\alpha(x,y)f(y,u(y))dy,
    \end{equation*}
    where $B_R=\{x\in\mathbb{R}^n\mid |x|<R\}$, and $G_R^\alpha(x,y)$ is the Green function for Laplacian $(-\Delta)^\frac{\alpha}{2}$ with $0<\alpha\leq 2$ in $B_R$, i.e.
    \begin{equation*}
        G_R^\alpha(x,y)=\begin{cases}
            \frac{C_{n,\alpha}}{|x-y|^{n-\alpha}}\int_0^\frac{(R^2-|x|^2)(R^2-|y|^2)}{R^2|x-y|^2} \frac{b^{\frac{\alpha}{2}-1}}{(1+b)^\frac{n}{2}} db,\;\;\;\;& x,y\in B_R,\\
            0,&x\text{ or }y\in B_R^c.
        \end{cases}
    \end{equation*}
By assumption $(f_3)$, we have $v_R\in C_{loc}^{1,1}(B_R\backslash\{0\})\cap C(\mathbb{R}^n)\cap \mathcal{L}_\alpha(\mathbb{R}^n)$ ($v_R\in C^2(B_R\backslash\{0\})\cap C(\mathbb{R}^n)$ if $\alpha=2$), and
\begin{equation*}
    \begin{cases}
        (-\Delta)^\frac{\alpha}{2} v_R(x)=f(x,u(x)),\;\;\;\;&x\in B_R\backslash\{0\},\\
        v_R(x)=0, &x\in B_R^c.
    \end{cases}
\end{equation*}
Let $w_R(x)=u(x)-v_R(x)$.According to Lemma 2.2 in \cite{dai2023liouville}, we can establish that
\begin{equation*}
\begin{cases}
    (-\Delta)^\frac{\alpha}{2} w_R(x)=0,\;\;\;\;&x\in B_R,\\
        w_R(x)\geq 0, &x\in B_R^c.
\end{cases}
\end{equation*}
According to the maximum principle for Laplacians and fractional Laplacians (as discussed in \cite{silvestre} and also in Theorem 2.1 of \cite{chen2}), we can conclude that
\begin{equation*}
    w_R(x)=u(x)-v_R(x)\geq 0,\;\;\;\;\forall x\in\mathbb{R}^n.
\end{equation*}
Fixing $x$ and letting $R\to \infty$, we have
\begin{equation*}
    u(x)\geq C_{n,\alpha}\int_{\mathbb{R}^n}\frac{f(y,u(y))}{|x-y|^{n-\alpha}} dy\geq 0.
\end{equation*}
Let $x=0$. By assumption $(f_3)$, 
\begin{equation}\label{eq 20}
   +\infty>u(0)\geq C_{n,\alpha}\int_{\mathbb{R}^n}\frac{f(y,u(y))}{|y|^{n-\alpha}} dy \geq C_{n,\alpha}\int_{\mathcal{C}}\frac{u^p(y)}{|y|^{n-\alpha-a}} dy.
\end{equation}
Let $v(x)=C_{n,\alpha}\int_{\mathbb{R}^n}\frac{f(y,u(y))}{|x-y|^{n-\alpha}} dy +C\geq C\geq 0.$. One can see $v(x)$ also solves
\begin{equation*}
    (-\Delta)^\frac{\alpha}{2} v(x)= f(x,u(x)) \;\;\;\;\forall x\neq 0.
\end{equation*}
Let $w(x)=u(x)-v(x)$. By \ Lemma 2.2 in \cite{dai2023liouville},
\begin{equation*}
\begin{cases}
    (-\Delta)^\frac{\alpha}{2} w(x)=0,\;\;\;\;&x\in \mathbb{R}^n,\\
        w(x)\geq 0, &x\in \mathbb{R}^n.
\end{cases}
\end{equation*}
By Liouville Theorem for Laplacians and fractional Laplacians (see Theorem 1 in \cite{zhuo2014liouville}),
we derive that
\begin{equation}\label{eq 19}
    u(x)=C_{n,\alpha}\int_{\mathbb{R}^n}\frac{f(y,u(y))}{|x-y|^{n-\alpha}} dy +C\geq C\geq 0.
\end{equation}
Combining (\ref{eq 19}) and (\ref{eq 20}), it follows that
\begin{equation*}
     C_{n,\alpha}\int_{\mathcal{C}}\frac{C^p}{|y|^{n-\alpha-a}} dy<+\infty.
\end{equation*}
This result implies that $C=0$ since $\mathcal{C}$ is an infinite cone. Therefore, (\ref{eq 19}) implies that
\begin{equation*}
    u(x)=C_{n,\alpha}\int_{\mathbb{R}^n}\frac{f(y,u(y))}{|x-y|^{n-\alpha}} dy.
\end{equation*}
Therefore, $u$ also solves the IE (\ref{main IE}).
\end{proof}

Moreover, we can derive a lower bound estimate of $u$ from the equivalent IE. To this end, we can see that for any $|x|\geq 1$,
\begin{equation}\label{lb 1}
    \begin{aligned}
    u(x) &\geq \frac{C_{n,\alpha}}{|x|^{n-\alpha}}\int_{\overline{B_{1/2}}\cap \mathcal{C}} f(y,u(y))dy\\
    &\overset{(f_3)}{\geq} \frac{C_{n,\alpha}}{|x|^{n-\alpha}}\int_{\overline{B_{1/2}}\cap \mathcal{C}} |y|^au^p(y) dy\\
    &\geq \frac{C_{n,\alpha}}{|x|^{n-\alpha}}.
\end{aligned}
\end{equation}

\section{Proof of Theorem \ref{main thm1}}

Next we will apply the \textit{method of scaling spheres }to give the following lower bounds for positive $u$, which will lead a contraction to IE (\ref{main IE}).

\begin{theorem}\label{main thm2}
      Assume $n>\alpha$, $0<\alpha\leq 2$, $0<p<p_c(a):=\frac{n+\alpha+2a}{n-\alpha}$. Suppose $f$ is subcritical and $f$ satisfies the assumption $(f_1)$, $(f_2)$ and $(f_3)$. Suppose $u$ is a non-trivial solution of (\ref{main PDE}). Then we have the following lower bound estimates.
\begin{enumerate}
    \item If $0<p<1$, then $u(x)\geq C_k|x|^k$ for any $ |x|\geq 1$ and any $k<\frac{\alpha+a}{1-p}$;
    \item If $1\leq p<p_c(a)=\frac{n+\alpha+a}{n-\alpha}$, then $u(x)\geq C_k|x|^k$ for any $ |x|\geq 1$ and any $k<+\infty$.
\end{enumerate}
\end{theorem}

\begin{proof}
    For any $\lambda>0$, define the Kelvin transform of $u$ by
    \begin{equation}\label{eq 7}
        u_\lambda(x)=\left(\frac{\lambda}{|x|}\right)^{n-\alpha} u\left(\frac{\lambda^2x}{|x|^2}\right)\;\;\;\;\forall x\neq 0.
    \end{equation}
    One may verify that $u_\lambda\in \mathcal{L}_\alpha(\mathbb{R}^n)\cap C_{loc}^{1,1}(\mathbb{R}^n_+)\cap C(\overline{\mathbb{R}^n_+}\backslash\{0\})$ for $0<\alpha<2$ and $u_\lambda\in C^{2}(\mathbb{R}^n_+)\cap C(\overline{\mathbb{R}^n_+}\backslash\{0\})$ for $\alpha=2$. It is well-known that
    \begin{equation}\label{eq 1}
        (-\Delta)^\frac{\alpha}{2}u_\lambda(x)=\left(\frac{\lambda}{|x|}\right)^{n+\alpha} ((-\Delta)^\frac{\alpha}{2}u)\left(\frac{\lambda^2x}{|x|^2}\right).
    \end{equation}
    Combine (\ref{eq 1}) and the assumption that f is subcritical, we deduce that for any $|x|<\lambda$,
    \begin{equation}\label{eq 2}
    \begin{aligned}
                (-\Delta)^\frac{\alpha}{2}u_\lambda(x)=\left(\frac{\lambda}{|x|}\right)^{n+\alpha} f\left(\frac{\lambda^2x}{|x|^2},u\left(\frac{\lambda^2x}{|x|^2}\right)\right) &= \left(\frac{\lambda}{|x|}\right)^{n+\alpha} f\left(\frac{\lambda^2}{|x|^2}x,\left(\frac{\lambda}{|x|}\right)^{-(n-\alpha)}u_\lambda(x)\right)\\
        &>f(x,u_\lambda(x)).
    \end{aligned}
    \end{equation}
    Before applying the method of scaling sphere, we give some definitions. Define $B_\lambda=\{x:|x|<\lambda\}$ and 
    \begin{equation*}
        w^\lambda(x)=u_\lambda(x)-u(x)\;\;\;\;\forall x\in B_\lambda\backslash\{0\}.
    \end{equation*}
By direct calculation, one can check that $w^\lambda$ is spherically anti-symmetric, i.e. 
\begin{equation*}
    w^\lambda(x)=-(w^\lambda)_\lambda(x).
\end{equation*}
The process of scaling a sphere can be broken down into two steps.

\textit{Step 1.} We will show that for $\lambda>0$ sufficiently small, it holds that $w_\lambda\geq 0$ for any $x\in B_\lambda\backslash\{0\}$.

Define 
\begin{equation*}
    B_\lambda^-=\{x\in B_\lambda\backslash\{0\}\mid w^\lambda(x)<0\}.
\end{equation*}
We will show through contradiction arguments that in fact $B_\lambda^-=\emptyset$ for $\lambda>0$ sufficiently small, which finishes the proof of \textit{Step 1.}

Suppose $B_\lambda^-\neq\emptyset$. For any $x\in B_\lambda$, by (\ref{eq 2}), we deduce that
\begin{equation*}
    (-\Delta)^\frac{\alpha}{2}w^\lambda(x)=(-\Delta)^\frac{\alpha}{2}u_\lambda(x)-(-\Delta)^\frac{\alpha}{2}u(x)>f(x,u_\lambda(x))-f(x,u(x))=c_\lambda(x)w^\lambda(x),
\end{equation*}
where 
\begin{equation*}
    c_\lambda(x):=\frac{f(x,u_\lambda(x))-f(x,u(x))}{u_\lambda(x)-u(x)}.
\end{equation*}
By $(f_1)$, it follows that $c_\lambda(x)\geq 0$ for any $x\in B_\lambda^-$. Since $0<u_\lambda(x)<u(x)\leq \sup_{B_\lambda} u<+\infty$, we use $(f_2)$ to get that
\begin{equation}
    c_\lambda(x)=|x|^{-\sigma} \frac{|x|^\sigma f(x,u_\lambda)-|x|^\sigma f(x,u)}{u_\lambda-u} \leq L_\lambda |x|^{-\sigma}\;\;\;\;\forall x\in B_\lambda^-,
\end{equation}
where $L_\lambda>0$ is dependent of $\lambda$. Thus
\begin{equation}\label{eq 5}
    (-\Delta)^\frac{\alpha}{2}w^\lambda(x)\geq L_\lambda |x|^{-\sigma}w^\lambda(x)\;\;\;\;\forall x\in B_\lambda^-.
\end{equation}
% By the definition (\ref{eq 7}) of $u_\lambda$, one may also notice that
% \begin{equation}\label{eq 6}
%     \liminf_{|x|\to 0}\left(u_\lambda(x)-u(x)\right)>0.
% \end{equation}
To show that $B_\lambda^-=\emptyset$, we need the following narrow region principle, and the idea comes from \cite{chen2017direct,dai2023liouville,cheng2017maximum,berestycki1991method,berestycki1988monotonicity}.
\begin{theorem}[Narrow region principle]\label{thm narrow}
    Assume $n>\alpha$, $0<\alpha\leq 2$ and $1\leq p<+\infty$. Let $\lambda>0$ and $A_{\lambda,l}=\{x\in\mathbb{R}^n\mid \lambda-l<|x|<\lambda\}$ be an annulus with thickness $l\in(0,\lambda)$. Let $B_\lambda^-$ be as above. Suppose $w^\lambda \in\mathcal{L}_\alpha(\mathbb{R}^n)\cap C_{loc}^{1,1}(A_{\lambda,l})$ for $0<\alpha<2$ and $w^\lambda \in C^2(A_{\lambda,l})$ for $\alpha=2$ and satisfies
    \begin{enumerate}
        \item\label{assump 1} $(-\Delta)^\frac{\alpha}{2}w^\lambda(x)-c_\lambda(x)w^\lambda(x)\geq 0$ in $A_{\lambda,l}\cap B_\lambda^-$, where $0\leq c_\lambda(x)\leq L_\lambda |x|^{-\sigma}$ with $\sigma<\alpha$;
        \item negative minimum of $w^\lambda$ is attained in $A_{\lambda,l}$.
    \end{enumerate}
    Then we have
    \begin{enumerate}
        \item[(i)] there exists a sufficiently small $\delta_0>0$ such that, if $0<\lambda\leq \delta_0$, then
        \begin{equation}\label{narrow eq 1}
            w^\lambda(x)\geq 0\;\;\;\;\forall x\in A_{\lambda,l};
        \end{equation}
        \item[(ii)] for arbitrarily fixed $\theta\in(0,1)$, there exists a sufficiently small $l_0>0$ depending on $\lambda$ continuously such that, if $0<l\leq l_0$ and $0<l\leq \theta \lambda$, then
        \begin{equation}\label{narrow eq 2}
            w^\lambda(x)\geq 0\;\;\;\;\forall x\in A_{\lambda,l}.
        \end{equation}
    \end{enumerate}
    % Furthermore, $w^\lambda(x)\geq 0$ for any $x\in B_\lambda\backslash\{0\}$. That is, $B_\lambda^-=\emptyset$.
\end{theorem}

\begin{proof}
    Suppose on the contrary that (\ref{narrow eq 1}) and (\ref{narrow eq 2}) do not hold. In other words, there exists a small $\delta_0$, $\lambda\in(0,\delta_0)$, and $x_1\in A_{\lambda,l}$ such that $w^\lambda(x_1)<0$. Also, for sufficiently small $l_0(\lambda)$, which depends continuously on $\lambda$, there exists $l\in(0,l_0(\lambda))$ and $x_2\in A_{\lambda,l}$ such that $w^\lambda(x_2)<0$.
    % We will show that, there exists $\lambda\in(0,\delta_0)$ with $\delta_0$ sufficiently small
    % For $0<\lambda<\delta_0$ sufficiently small and $0<l<l_0(\lambda)$ sufficiently small we will get a contradiction. 
    % By assumption (2)(3) and our hypothesis, there exists $\Tilde{x}\in A_{\lambda,l}\cap B_\lambda^-$ such that
    % \begin{equation*}
    %     w^\lambda(\Tilde{x})= \min_{A_{\lambda,l}\cap B_\lambda^-} w^\lambda(x)= \min_{B_\lambda\backslash\{0\}} w^\lambda(x)<0.
    % \end{equation*}
    
    By the same process in \cite{dai2023liouville}, Theorem 2.8, it follows that there exists $\Tilde{x}\in A_{\lambda,l}\cap B_\lambda^-$ such that
    \begin{equation}\label{eq 3}
        (-\Delta)^\frac{\alpha}{2}w^\lambda (\Tilde{x}) \leq \frac{C}{l^\alpha}w^\lambda (\Tilde{x})<0
    \end{equation}
    for $0<\alpha\leq 2$.
    (\ref{eq 3}) and assumption \ref{assump 1} yield that
    \begin{equation}\label{eq 4}
        L_\lambda|\Tilde{x}|^{-\sigma}\geq c_\lambda(\Tilde{x}) \geq \frac{C}{l^\alpha}.
    \end{equation}
    We can derive from (\ref{eq 4}) that for $\sigma\in(0,\alpha)$,
    \begin{equation*}
        C\leq L_\lambda\frac{l^\alpha}{|\Tilde{x}|^{\sigma}}\leq L_\lambda\frac{\lambda^\alpha}{|\lambda-l|^{\sigma}},
    \end{equation*}
    and for $\sigma\in(-\infty,0]$, 
    \begin{equation*}
        C \leq L_\lambda\frac{l^\alpha}{|\Tilde{x}|^{\sigma}} \leq L_\lambda\lambda^{\alpha-\sigma}.
    \end{equation*}
    In both cases, we encounter a contradiction when $l$ is fixed and $\lambda$ is sufficiently small. Thus, (i) is verified. 
    
%     To verify (ii), we use the same method as (2.18) in Theorem 2.2, \cite{chen2017direct} to derive that
% \begin{equation}\label{eq 21}
%     (-\Delta)^{\frac{\alpha}{2}} w^\lambda(\Tilde{x})\leq w^\lambda(\Tilde{x})\int_{B_\lambda^c}\frac{dy}{|\Tilde{x}-y|^{n+\alpha}}.
% \end{equation}
%     % Denote $d=dist(\Tilde{x},B_\lambda)$, then $0<d<l$.
%     By (\ref{eq 21}), it follows that
%     \begin{equation}\label{eq 22}
%     \begin{aligned}
%         (-\Delta)^{\frac{\alpha}{2}} w^\lambda(\Tilde{x}) \leq C w^\lambda(\Tilde{x})\int_{B_\lambda^c\cap (B_{4l_1}(\Tilde{x})\backslash B_{l_1}(\Tilde{x}))}\frac{dy}{|\Tilde{x}-y|^{n+\alpha}}\leq C w^\lambda(\Tilde{x})\int_{l_1}^{4l_1}\frac{dr}{r^{\alpha+1}} \leq C \frac{w^\lambda(\Tilde{x})}{l_1^\alpha}.
%     \end{aligned}
%     \end{equation}
%    (\ref{eq 22}) and assumption \ref{assump 1} yield that
Similarly, 
    \begin{equation*}
        \begin{aligned}
            C\leq L_\lambda\frac{l^\alpha}{((1-\theta)\lambda)^\sigma}\;\;\;\;\text{for }\sigma\in(0,\alpha),\\
            C\leq L_\lambda\frac{l^\alpha}{\lambda^\sigma}\;\;\;\;\text{for }\sigma\in(-\infty,0].
        \end{aligned}
    \end{equation*}
    And we get a contraction when $l_0$ is sufficiently small and $l\in(0,\min\{\theta\lambda,l_0\})$. This verifies (ii).
    % Furthermore, the fact that $w^\lambda\geq 0$ in $A_{\lambda,l}$ can actually imply that
    % \begin{equation*}
    %     w^\lambda\geq 0 \;\;\;\;\text{in }B_\lambda\backslash\{0\}.
    % \end{equation*}
    % Indeed, suppose that $w^\lambda$ has a negative value somewhere in $(B_\lambda\backslash\{0\})\backslash A_{\lambda,l}$. According to (\ref{eq 6}), $w^\lambda$ attains negative minimum in $(B_\lambda\backslash\{0\})\backslash A_{\lambda,l}$, which contradicts assumption (3).
\end{proof}
We claim that there exists a sufficiently small $\eta_0>0$ such that the following argument holds: if $0<\lambda<\eta_0$, then there exists a sufficiently small $\theta>0$ such that
\begin{equation}\label{eq 23}
    w^\lambda(x)\geq 1,\;\;\;\;\forall x\in \overline{B_{\theta\lambda}}\backslash\{0\}.
\end{equation}
In fact, we hold $\lambda$ fixed and choose $\theta > 0$ small enough such that $\frac{\lambda^2x}{|x|^2}$ attains large values. Thus by the estimate (\ref{lb 1}), for any $x\in \overline{B_{\theta\lambda}}\backslash\{0\}$, we have
\begin{equation*}
    u_\lambda(x)=(\frac{\lambda}{|x|})^{n-\alpha} u(\frac{\lambda^2x}{|x|^2}) \geq C (\frac{\lambda}{|x|})^{n-\alpha} (\frac{|x|}{\lambda^2})^{n-\alpha}=\frac{C}{\lambda^{n-\alpha}}.
\end{equation*}
Then for $\lambda$ sufficiently small, it holds that
\begin{equation*}
    w^\lambda(x)\geq \frac{C}{\lambda^{n-\alpha}}-C\geq 1.
\end{equation*}
As a result, (\ref{eq 23}) holds.

Now let $\epsilon_0=\min\{\delta_0, \eta_0\}$, $l=(1-\theta)\lambda$. For $0<\lambda\leq \epsilon_0$, 
(\ref{eq 5}) and (\ref{eq 23}) imply that the assumptions in Theorem \ref{thm narrow} hold. Consequently, by Theorem \ref{thm narrow} (i) we conclude that
\begin{equation*}
    w^\lambda\geq 0\;\;\;\;\text{in }A_{\lambda,l}.
\end{equation*}
Hence,
\begin{equation*}
    B_\lambda^-=\emptyset\;\;\;\;\text{for $\lambda>0$ sufficiently small}.
\end{equation*}
This completes \textit{Step 1.}

\textit{Step 2.} We will dilate the sphere $\{x\in\mathbb{R}^n\mid |x|=\lambda\}$ outward until $\lambda=+\infty$ to derive lower bound estimates on $u$.

% A starting point is provided in \textit{Step 1} to dilate the sphere from near $\lambda=0$.

In fact, \textit{Step 1} provides a starting point to carry out the method of scaling spheres. Then we will continuously increase $\lambda$ as long as
\begin{equation*}
    w_\lambda\geq 0\;\;\;\;\text{in }B_\lambda\backslash\{0\}.
\end{equation*}
Define 
\begin{equation*}
    \lambda_0=\sup\{\lambda>0\mid w^\mu\geq 0 \text{ in }B_\mu\backslash\{0\},\;\forall 0<\mu\leq \lambda\}.
\end{equation*}
It implies that
\begin{equation*}
    w^{\lambda_0}\geq 0\;\;\;\;\forall x\in B_{\lambda_0}\backslash\{0\}.
\end{equation*}
We will show that $\lambda_0=+\infty$. 

Suppose, to the contrary, that $\lambda_0<+\infty$. To establish a contradiction, we will prove the following claim:
\begin{equation}\label{eq 8}
     w^{\lambda_0}\equiv 0\;\;\;\;\forall x\in B_{\lambda_0}\backslash\{0\}.
\end{equation}

We use a contradiction argument to show the claim. Suppose, to the contrary,  that (\ref{eq 8}) does not hold. Then $w^{\lambda_0}$ is nonnegative in $B_{\lambda_0}\backslash\{0\}$, and there exists $x_0\in B_{\lambda_0}\backslash\{0\}$ such that $w^{\lambda_0}(x_0)>0$. We will first prove that 
\begin{equation}\label{eq 11}
    w^{\lambda_0}>0 \;\;\;\;\text{in } B_{\lambda_0}\backslash\{0\}, 
\end{equation}
and then apply Theorem \ref{thm narrow} to derive a contradiction.

Since the solution $u$ of PDE (\ref{main PDE}) also satisfies IE (\ref{main IE}), we can get through direct calculations that
\begin{equation}\label{eq 9}
    u(x)=\int_{B_{\lambda_0}} \frac{f(y,u(y))}{|x-y|^{n-\alpha}}+ \frac{f \left(\frac{\lambda_0^2}{|y|^2}y, u(\frac{\lambda_0^2}{|y|^2}y) \right)}{|\frac{|y|}{\lambda_0}x-\frac{\lambda_0}{|y|}y|^{n-\alpha}} dy.
\end{equation}
Apply (\ref{main IE}) and the equality $|\frac{x}{|x|^2}-\frac{z}{|z|^2}|=\frac{|x-z|}{|x||z|}$, we can derive that
\begin{equation*}
    u_{\lambda_0}(x)=\int_{\mathbb{R}^n} \frac{f \left(\frac{\lambda_0^2}{|y|^2}y, u(\frac{\lambda_0^2}{|y|^2}y) \right)}{|x-y|^{n-\alpha}} \left(\frac{\lambda_0}{|y|}\right) ^{n+\alpha} dy.
\end{equation*}
Using a similar calculation, we obtain
\begin{equation}\label{eq 10}
     u_{\lambda_0}(x)=\int_{B_{\lambda_0}} \frac{f(y,u(y))}{|\frac{|y|}{\lambda_0}x-\frac{\lambda_0}{|y|}y|^{n-\alpha}}+ \left(\frac{\lambda_0}{|y|}\right) ^{n+\alpha} \frac{f \left(\frac{\lambda_0^2}{|y|^2}y, u(\frac{\lambda_0^2}{|y|^2}y) \right)}{|x-y|^{n-\alpha}} dy.
\end{equation}
Combine (\ref{eq 9}) and (\ref{eq 10}), we obtain that for any $x\in B_{\lambda_0}\backslash\{0\}$,
\begin{equation*}
    w^{\lambda_0}(x)= \int_{B_{\lambda_0}} \left( \frac{1}{|x-y|^{n-\alpha}}- \frac{1}{|\frac{|y|}{\lambda_0}x-\frac{\lambda_0}{|y|}y|^{n-\alpha}}\right) \left( \left(\frac{\lambda_0}{|y|}\right)^{n+\alpha} f \left(\frac{\lambda_0^2}{|y|^2}y, u(\frac{\lambda_0^2}{|y|^2}y) \right)-f(y,u(y)) \right)dy.
\end{equation*}
One may check that 
\begin{equation*}
    \frac{1}{|x-y|^{n-\alpha}}- \frac{1}{|\frac{|y|}{\lambda_0}x-\frac{\lambda_0}{|y|}y|^{n-\alpha}}>0\;\;\;\;\;\forall x,y\in B_{\lambda_0}\backslash\{0\}.
\end{equation*}
Since $f(x,u(x))$ is subcritical and satisfies assumption $(f_1)$, it follows that
\begin{equation}\label{eq 16}
    w^{\lambda_0}(x)> \left( \frac{1}{|x-y|^{n-\alpha}}- \frac{1}{|\frac{|y|}{\lambda_0}x-\frac{\lambda_0}{|y|}y|^{n-\alpha}}\right) \big( f(y,u_{\lambda_0}(y))-f(y,u(y))\big) dy\geq 0\;\;\;\;\forall x\in B_{\lambda_0}\backslash\{0\}.
\end{equation}
Thus (\ref{eq 11}) is verified.

Next, we claim that there exists $\epsilon>0$ such that $w^\lambda(x)\geq0$ in $B_\lambda\backslash\{0\}$ for all $\lambda\in[\lambda_0,\lambda_0+\epsilon)$. Consequently, this implies that the sphere can be dilated outward slightly more than $\lambda_0$, which contradicts the definition of $\lambda_0$.

Define 
\begin{equation*}
    \Tilde{l_0}=\min\{\theta\lambda,\min_{\lambda\in[\lambda_0,2\lambda_0]} l_0(\lambda)\}>0,
\end{equation*}
where $l_0(\lambda)$ and $\theta$ are given in Theorem \ref{thm narrow} (ii). For a fixed small $0<r_0<\frac{1}{2}\min\{\Tilde{l_0},\lambda_0\}$, define 
\begin{equation*}
    m_0=\inf_{x\in\overline{B_{\lambda_0-r_0}}\backslash\{0\}} w^{\lambda_0}(x).
\end{equation*}
By (\ref{eq 11}), $m_0>0$. Using the same uniform-continuity arguments as in the proof of Theorem 1.3, Step 2 in \cite{dai2023liouville}, we can establish the existence of a sufficiently small $\epsilon_1$ in the interval $(0, \frac{1}{2}\min{\Tilde{l_0},\lambda_0})$. This small value of $\epsilon_1$ ensures that for all $\lambda$ in the range $[\lambda_0, \lambda_0+\epsilon_1]$, the following inequality holds:
\begin{equation}\label{eq 14}
w^\lambda(x)\geq \frac{m_0}{2}>0\;\;\;\;\forall 
 0<|x|\leq \lambda_0-r_0.
\end{equation}
That means 
\begin{equation}\label{eq 12}
    w^\lambda \text{ attains negative minimum in } A_{\lambda,l},\;\text{where }l=\lambda-\lambda_0+r_0.
\end{equation}
% One can also observe that the following conditions hold: 
% \begin{equation}\label{eq 13}
%     \begin{cases}
%         w^\lambda(x)=0\;\;\;\;\text{on }B_\lambda,\\
%         \liminf_{x\to 0}w^\lambda(x)>0.
%     \end{cases}
% \end{equation}
For any $\lambda\in[\lambda_0, \lambda_0+\epsilon_1]$, let $l=\lambda-\lambda_0+r_0$. Then $l\in(0,\Tilde{l_0})$. By combining (\ref{eq 12}) and (\ref{eq 5}), one can see that the assumptions in Theorem \ref{thm narrow} are satisfied. Thus,
\begin{equation}\label{eq 15}
    w^\lambda\geq 0,\;\;\;\;\forall x\in A_{\lambda,l}=\{x\in\mathbb{R}^n\mid \lambda_0-r_0<|x|<\lambda\}.
\end{equation}
(\ref{eq 14}) and (\ref{eq 15}) imply that for any $\lambda\in [\lambda_0, \lambda_0+ \epsilon_1]$,
\begin{equation*}
     w^\lambda\geq 0,,\;\;\;\;\forall x\in B_\lambda\backslash\{0\},
\end{equation*}
which contradicts with the definition of $\lambda_0$. As a result, (\ref{eq 8}) is satisfied, implying that
\begin{equation*}
    w^{\lambda_0}\equiv 0\;\;\;\;\forall x\in B_{\lambda_0}\backslash\{0\}.
\end{equation*}
However, this contradicts (\ref{eq 16}). Therefore, we can conclude that
\begin{equation*}
    \lambda_0=+\infty.
\end{equation*}
That is, for all $\lambda\in(0,+\infty)$, 
\begin{equation*}
    u(x)\leq \left( \frac{\lambda}{|x|} \right)^{n-\alpha}u(\frac{\lambda^2x}{|x|^2})\geq u(x),\;\;\;\;\forall 0<|x|\leq \lambda.
\end{equation*}
And it is equivalent to that for all $\lambda\in(0,+\infty)$
\begin{equation*}
    u(x)\geq \left( \frac{\lambda}{|x|} \right)^{n-\alpha}u(\frac{\lambda^2x}{|x|^2}) ,\;\;\;\;\forall|x|\geq \lambda.
\end{equation*}
For $|x|\geq 1$, let $\lambda=\sqrt{x}$. Then 
\begin{equation*}
    u(x)\geq \frac{1}{|x|^\frac{n-\alpha}{2}} u(\frac{x}{|x|}),\;\;\;\;\forall|x|\geq 1.
\end{equation*}
So,
\begin{equation}\label{eq 17}
    u(x)\geq (\min_{|x|=1}u(x))\frac{1}{|x|^\frac{n-\alpha}{2}} =\frac{C_0}{|x|^\frac{n-\alpha}{2}},\;\;\;\;\forall|x|\geq 1.
\end{equation}
Let $\mu_0=\frac{n-\alpha}{2}$. The assumption $(f_3)$ and (\ref{eq 17}) yield that for any $|x|\geq 1$,
\begin{equation}\label{eq 18}
\begin{aligned}
        u(x) & = \int_{\mathbb{R}^n}\frac{f(y,u(y))}{|x-y|^{n-\alpha}} dy\\
    &\geq C\int_{\{ 2|x|\leq |y| \leq 4|x|\}\cap \mathcal{C}} \frac{|y|^au^p(y)}{|x-y|^{n-\alpha}} dy\\
    &\geq \frac{C}{|x|^{n-\alpha}} \int_{\{ 2|x|\leq |y| \leq 4|x|\}\cap \mathcal{C}}\frac{1}{|y|^{p\mu_0-\alpha}} dy\\
    &\geq \frac{C}{|x|^{n-\alpha}} \int_{\{ 2|x|\leq |y| \leq 4|x|\}} \frac{dy}{|y|^{p\mu_0-\alpha}}\\
    &\geq \frac{C_1}{|x|^{p\mu_0-(a+\alpha)}}.
\end{aligned}
\end{equation}
Let $\mu_1=p\mu_0-(a+\alpha)$. Since $0<p<p_c(a):=\frac{n+\alpha+2a}{n-\alpha}$, it follows that $\mu_1<\mu_0$. Thus, we have obtained a better lower bound estimate than (\ref{eq 17}) after one iteration. For $k=0,1,2,...$, define
\begin{equation*}
    \mu_{k+1}=p\mu_k-(a+\alpha).
\end{equation*}
Since $p<\frac{n+\alpha+2a}{n-\alpha}$, one can see that $\{\mu_k\}$ is a decreasing sequence, and
\begin{equation*}
\begin{aligned}
      & \mu_k\to -\frac{a+\alpha}{1-p} \;\;\;\;\text{if }0<p<1;\\
   & \mu_k\to -\infty\;\;\;\;\text{if }1\leq p<\frac{n+\alpha+2a}{n-\alpha}. 
\end{aligned}
\end{equation*}
Thus, continuing the iteration process as (\ref{eq 18}), we have the following lower bound estimates:
\begin{enumerate}
    \item If $0<p<1$, then $u(x)\geq C_k|x|^k$ for any $ |x|\geq 1$ and any $k<\frac{\alpha+a}{1-p}$;
    \item If $1\leq p<p_c(a)=\frac{n+\alpha+a}{n-\alpha}$, then $u(x)\geq C_k|x|^k$ for any $ |x|\geq 1$ and any $k<+\infty$.
\end{enumerate} 
This finishes Theorem \ref{main thm2}.
\end{proof}
Now we choose $\Tilde{x}\in\mathcal{C},
\;|\Tilde{x}|=1$, where $\mathcal{C}$ is give in assumption $(f_3)$. One can see that the lower bound estimates in Theorem \ref{main thm2} actually contradicts the following intergrability indicated by IE (\ref{main IE}), that is,
\begin{equation*}
\begin{aligned}
    +\infty>u(\Tilde{x})&\geq C\int_{\{|y| \geq 2\}\cap \mathcal{C}}\frac{|y|^au^p(y)}{|\Tilde{x}-y|^{n-\alpha}} dy\\
&    \geq C\int_{\{|y|\geq 2\}\cap \mathcal{C}}\frac{u^p(y)}{|y|^{n-\alpha-a}} dy.
\end{aligned}
\end{equation*}
Therefore, $u$ must be trivial, that is, $u\equiv 0$ in $\mathbb{R}^n$.

\noindent COI: The authors declared that they had no conflict of interest.
\bibliographystyle{siam}
\bibliography{ref}
\end{document}